\documentclass[12pt]{article}

\usepackage{amsmath,amsthm}
\usepackage{amsfonts,xcolor,graphicx,url}
\usepackage{epstopdf,epsfig}

\usepackage[margin=2.5cm]{geometry}

\usepackage{titlesec}
\titleformat{\section}
{\normalfont\bfseries}{\thesection}{1em}{}
\titleformat{\subsection}
{\normalfont}{\thesubsection}{1em}{}

\numberwithin{equation}{section}

\theoremstyle{plain}
\newtheorem{theorem}{Theorem}
\newtheorem{lemma}[theorem]{Lemma}
\newtheorem{proposition}[theorem]{Proposition}

\theoremstyle{definition}
\newtheorem{definition}[theorem]{Definition}

\theoremstyle{remark}
\newtheorem{remark}[theorem]{Remark}

\begin{document}

\title{
Stability of piecewise flat Ricci flow
}

\author{Rory Conboye}

\date{}
\maketitle

\let\thefootnote\relax\footnotetext{
\hspace{-0.75cm}
	Department of Mathematics and Statistics
	 \\
	American University
	 \\
	4400 Massachusetts Avenue, NW
	 \\
	Washington, DC 20016, USA
}
\let\thefootnote\svthefootnote

\begin{abstract}
	The stability of a recently developed piecewise flat Ricci flow is investigated, using a linear stability analysis and numerical simulations, and a class of piecewise flat approximations of smooth manifolds is adapted to avoid an inherent numerical instability.
	These adaptations have also been used in a related paper to show the convergence of the piecewise flat Ricci flow to known smooth Ricci flow solutions for a variety of manifolds.
	
	\
	
	
	\
	
	
\end{abstract}

\section{Introduction}


The Ricci flow is a uniformizing flow on  manifolds, evolving the metric to reduce the strength of the Ricci curvature. It was initially developed by Richard Hamilton to help prove the Thurston geometrization conjecture \cite{HamRF}, and it remains an important tool for analysing the interplay between geometry and topology. Recently, its use has expanded further, with numerical evolutions finding applications in facial recognition \cite{GuZengFacial}, cancer detection \cite{GuCancer}, and space-time physics \cite{WoolgarRFphys,WiseRF-BH,WiseRF-CFT}.


Piecewise flat manifolds are formed by joining flat Euclidean segments in a consistent manner.
To allow for a natural refinement, the piecewise flat manifolds considered here are formed from a mesh of cube-like blocks, with each block composed of six flat tetrahedra. 
In this case, the topology is determined by the simplicial graph, and the geometry by a discrete set of edge-lengths.
This conveniently leads to a piecewise flat Ricci flow as a rate of change of edge-lengths, as introduced in \cite{PLCurv}, with the rate of change determined by an approximation of the smooth Ricci curvature at the edges.
However, despite the curvature approximations and some analytic solutions of the piecewise flat Ricci flow converging to their smooth counterparts as the mesh is refined \cite{PLCurv}, numerical evolutions have led to an exponential growth of errors in the edge-lengths.


This instability is analysed here, and a method for suppressing the instability proposed, using blocks that are internally flat instead of being composed of flat tetrahedra. This is implemented by introducing a constraint on the length of an edge in the interior of each block. 
A linear stability analysis and numerical simulations are then used to demonstrate the initial instability and the effectiveness of its suppression. 
Computations using these adapted piecewise flat manifolds can already be seen in \cite{PLRF}, where they have successfully been used in the piecewise flat Ricci flow of manifolds with a variety of different properties, showing convergence to known Ricci flow solutions and behaviour.


The remainder of the paper begins with an introduction to the piecewise flat manifolds used in this paper, and the piecewise flat Ricci flow, summarizing the main results from \cite{PLCurv}. The numerical instabilities are described in section \ref{sec:stabMethod}, along with a motivation and explanation of the suppression. A linear stability analysis in section \ref{sec:Instability} shows the exponential growth to arise from the numerical errors in the length measurements, with numerical simulations matching the growth rates. In section \ref{sec:Stability}, the same linear stability analysis no longer indicates an exponential growth when the suppressing method is used, with numerical simulations also showing stable behaviour.

\section{Background}
\label{sec:Prelims}

\subsection{Piecewise flat manifolds and triangulations}
\label{sec:Tri}

In three dimensions, the most simple of piecewise flat manifolds are formed by joining Euclidean tetrahedra together, with the triangular faces between neighbouring tetrahedra identified. The resulting graph encodes the topology of the manifold, with the geometry completely determined by the set of edge-lengths. 
Piecewise flat approximations of smooth manifolds can be constructed by first setting up a tetrahedral graph on the smooth manifold, using geodesic segments as edges. A piecewise flat manifold can then be defined using the same graph, with the edge-lengths determined by the lengths of the corresponding geodesic segments on the smooth manifold. Such a piecewise flat approximation is known as a triangulation of the smooth manifold.

In order to test for convergence to smooth curvature and Ricci flow, a set of triangulations which can be scaled in some regular way must be used. Three such triangulation-types were defined in \cite{PLRF}, using building blocks which can be tiled to form a complete tetrahedral graph. These building blocks are defined below, in terms of a set of coordinates $x$, $y$, and $z$, with each block covering a unit of volume. Diagrams of each block are also shown in figure \ref{fig:blocks}.

\begin{figure}[h]
	\centering
	\includegraphics[scale=1.2]{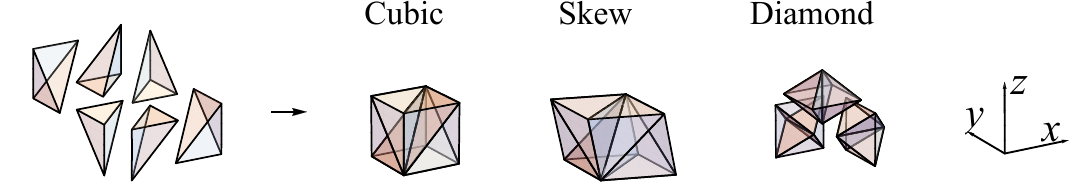}
	\caption{The three different block types, with the six tetrahedra of the cubic block on the far left, and a slight separation of the three diamond shapes forming the diamond block.}
	\label{fig:blocks}
\end{figure}

\begin{enumerate}
	\item The \emph{cubic} block forms a coordinate cube composed of six tetrahedra, with three independent edges along the coordinate directions, three independent face-diagonals and a single internal body-diagonal. The tetrahedra are specified so that the face-diagonals on opposite sides are in the same directions.
	
	\item The \emph{skew} block has the same structure as the cubic block, but with its vertices skewed in the $x$ and $z$ directions, with $v_x = (1, -1/3, 0)$ and $v_z = (-1/3, -2/9, 1)$.
	
	\item  The \emph{diamond} block is constructed from a set of four tetrahedra around each coordinate axis, with edges in the outer rings formed by the remaining coordinate directions.
\end{enumerate}

These blocks can be used to triangulate manifolds with three-torus ($T^3$) topologies, using a cuboid-type grid of blocks to cover the fundamental domain, identifying the triangles, edges and vertices on opposite sides. The resulting triangulations have computational domains that are compact without boundary. Manifolds with other topologies can also be triangulated using these blocks but require slightly more complicated tetrahedral graphs, see for example the Nil manifold in \cite{PLRF}. Since the stability should depend on the local structure of the piecewise flat manifold, only $T^3$ topology triangulations will be considered here.

\subsection{Piecewise flat curvature and Ricci Flow}
\label{sec:PFRF}


While any neighbouring pair of tetrahedra in a piecewise flat manifold still forms a Euclidean space, a natural measure of curvature arises from the sum of the dihedral angles $\theta_t$ around an edge $\ell$, with the difference from $2 \pi$ radians known as a deficit angle,
\begin{equation}
\epsilon_\ell := 2 \pi - \sum_t \theta_t \, .
\end{equation}
Triangulations of smooth manifolds are deemed good approximations if the deficit angles are uniformly small. The resolution of a piecewise flat approximation can then be increased by having a higher concentration of tetrahedra, or a finer grid of the blocks defined above.

\begin{figure}[h]
	\centering
	\includegraphics[scale=1]{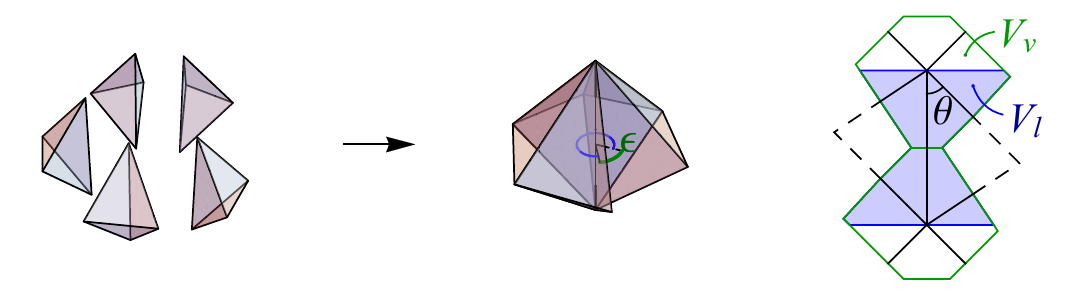}
	\caption{The deficit angle at an edge, and a cross section of the region around an edge showing the vertex and edge volumes.}
	\label{fig:tet}
\end{figure}

Conceptually, the deficit angles correspond to surface integrals of the sectional curvature orthogonal to each edge. However, a number of examples show that a single deficit angle does not carry enough information to approximate the smooth curvature directly, see section 5.1 of \cite{PLCurv} for example. Instead, this correspondence is used to construct volume integrals of both the scalar curvature at each vertex and the sectional curvature orthogonal to each edge, which can then be used to give the Ricci curvature along the edges.


Volumes $V_v$ associated with each vertex $v$ are defined to form a dual tessellation of the piecewise flat manifold, with barycentric duals used here. Edge-volumes $V_\ell$ are then defined as a union of the volumes $V_v$ for the vertices on either end, capped by surfaces orthogonal to the edges at each vertex, as shown in figure \ref{fig:tet}. From \cite{PLCurv}, piecewise flat approximations of the scalar curvature at each vertex $v$, sectional curvature orthogonal to each edge $\ell$, and the Ricci curvature along each edge $\ell$, are given by the expressions:
\begin{subequations}
\label{eq:RandK}
\begin{align}
R_v \
&= \frac{1}{V_v} \sum_{i} |\ell_i| \, \epsilon_i \, ,
\label{eq:R_v} \\
K_\ell^\perp
&= \frac{1}{V_\ell} \left(
	|\ell| \, \epsilon_\ell + \sum_i \frac{1}{2} |\ell_i| \cos^2 (\theta_i) \epsilon_i
\right) ,
\label{eq:K_l} \\
Rc_\ell
&= \frac{1}{4} (R_{v_1} + R_{v_2}) - K_\ell^\perp \, ,
\label{eq:Rc_l}
\end{align}
\end{subequations}
with the indices $i$ labelling the edges intersecting the volumes $V_v$ and $V_\ell$, $\theta_i$ representing the angle between the edge $\ell_i$ and $\ell$, and $v_1$ and $v_2$ indicating the vertices bounding $\ell$. Computations for a number of manifolds have shown these expressions to converge to their corresponding smooth values \cite{PLCurv}. Similar constructions have been developed for the extrinsic curvature \cite{PLExCurv}, with numerical computations successfully converging to their smooth values.


The Ricci flow of a smooth manifold changes the metric $g$ due to the Ricci curvature $Rc$, 
\begin{equation}
\frac{d g}{d t} = - 2 Rc \, .
\end{equation}
The resulting change in the length of a geodesic segment can be given solely by the Ricci curvature along and tangent to it, as shown in section 6.3 of \cite{PLCurv}. Since the edge-lengths of a triangulation correspond to the lengths of these geodesic segments, a piecewise flat approximation of the smooth Ricci flow can be given by a fractional change in the edge-lengths,
\begin{equation}
\label{eq:PFRF}
\frac{1}{|\ell|} \frac{d |\ell|}{d t}
= - Rc_\ell
\end{equation}
The equation above has been shown to converge to known smooth Ricci flow solutions as the resolution is increased, using analytic computations for symmetric manifolds in \cite{PLCurv}, and numerical evolutions for a variety of other manifolds in \cite{PLRF}. This approach has also been used by Alsing, Miller and Yau \cite{AlsingMillerYau18}, but with a different edge volume $V_\ell$, which works when the triangulations are adapted to the spherical symmetry of the manifolds studied there.

\section{Instability and suppression}
\label{sec:stabMethod}

\subsection{Instability}

Despite the close approximation of the piecewise flat Ricci curvature $Rc_\ell$ to its corresponding smooth values \cite{PLCurv}, initial numerical evolutions of the piecewise flat Ricci flow resulted in an exponential growth of the face-diagonals for both the cubic and skew triangulation types, even for manifolds that are initially flat. 
Since the deficit angles should all be zero for a triangulation of a flat manifold, this growth must arise from numerical errors.

The top two graphs in figure \ref{fig:NumErr} show how far the face-diagonal edge-lengths deviate from a flat triangulation. 
These deviations start at the level of the numerical precision and grow exponentially from there. 
The growth is invariant to the step size, has occurred 
for every grid size tested and 
for both the normalized and non-normalized piecewise flat Ricci flow equations, and soon dominates any evolution. The growth rates also \emph{increase} when the scale of the edge-lengths is reduced, countering any improved precision from an increase in resolution. This has led to the proposition below.

\begin{proposition}
	\label{prop:instab}
	The piecewise flat Ricci flow is exponentially unstable when directly applied to cubic and skew type triangulations.
\end{proposition}

This proposition is proved in section \ref{sec:ProofInstab} using a linear stability analysis of all cubic and skew triangulations of a three-torus. Section \ref{sec:NumSim} then shows the growth rates for a number of numerical simulations to be in close agreement with the results of the linear stability analysis.

\begin{figure}[h!]
	\begin{center}
		\includegraphics[scale=1]{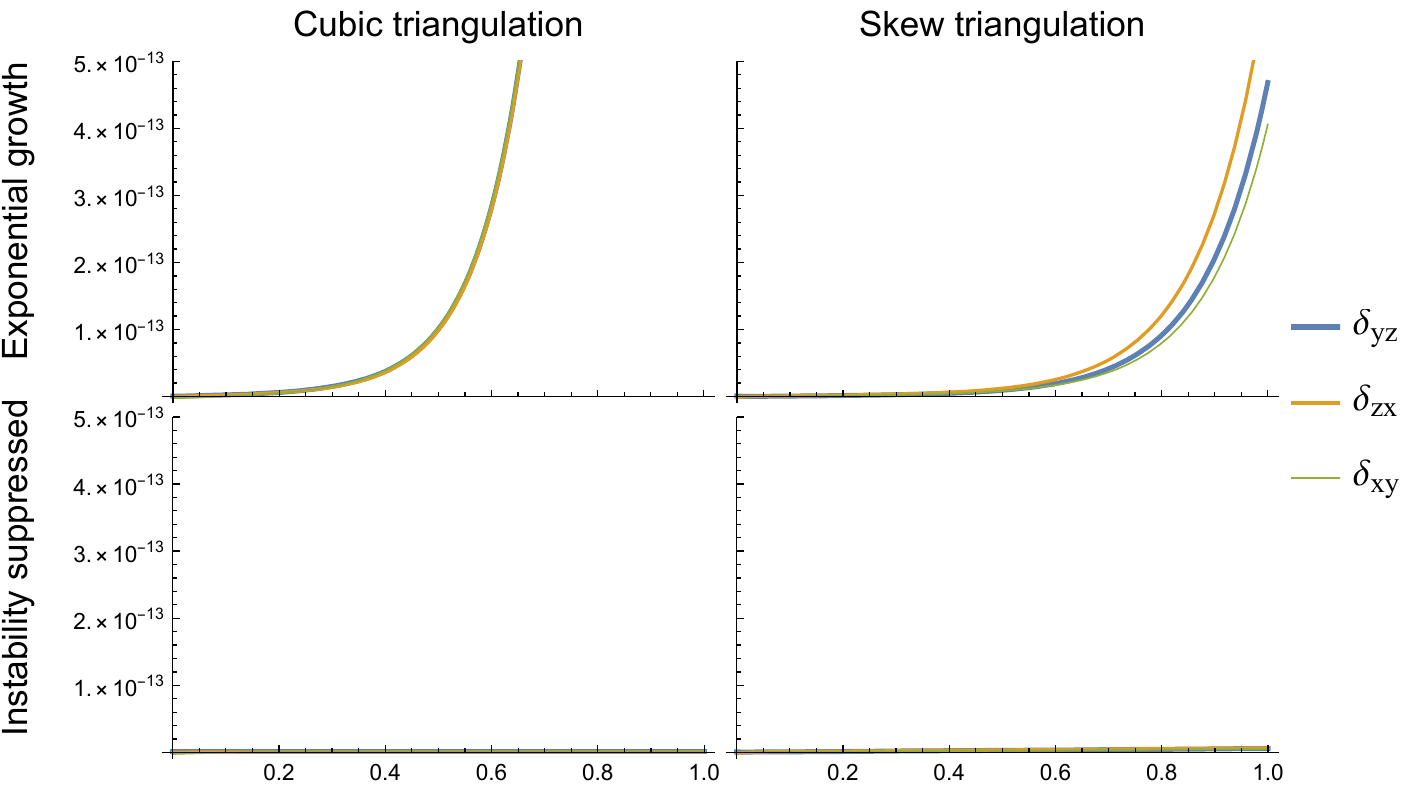}
	\end{center}
	\vspace{-0.5cm}
	\caption{
	Exponential growth of the errors in the face-diagonals is shown on the top row, for both cubic and skew triangulations of a flat three-torus. The bottom shows the suppression of this growth when the body-diagonals are adjusted to give blocks with flat interiors.
	}
	\label{fig:NumErr}
\end{figure}

\subsection{Suppressing instability}

While it is clear that the evolution shown on the top row of figure \ref{fig:NumErr} does not correspond to smooth Ricci flow, it is also inherently non-smooth in nature, particularly with the growth rates increasing as the resolution of a triangulation is increased. This suggests an extra freedom in the cubic and skew triangulations that does not arise in smooth manifolds.

In both the linearized equations in section \ref{sec:ProofInstab}, and the numerical simulations in section \ref{sec:NumSim}, all of the face-diagonal edges can be seen to grow at the same growth rate, with the other types of edges remaining unchanged. 
With only the face-diagonals growing, each face can still be viewed as a flat parallelogram. The exterior of each block can also still be embedded in three dimensional Euclidean space as a parallelepiped, with the change in the face-diagonals acting similar to a change of coordinates. This is shown in the left two images of figure \ref{fig:BodyDiag}.

\begin{figure}[h]
	\centering
	\includegraphics[scale=1.2]{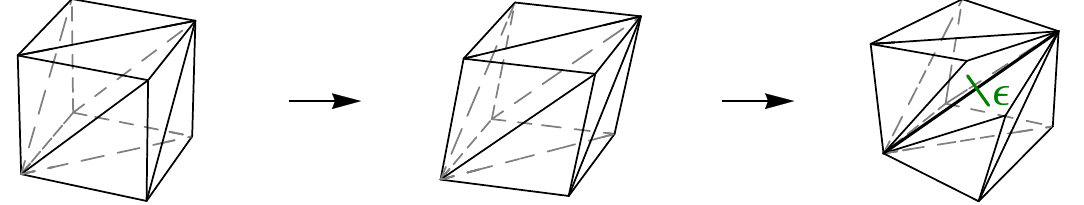}
	\caption{
		The effect of the face-diagonal growth on the exterior of a cubic block, with the resulting deficit angle arising from an unchanged body-diagonal shown on the right.
		}
	\label{fig:BodyDiag}
\end{figure}

The distance between the bottom left and top right vertices of the parallelepiped in figure \ref{fig:BodyDiag} must clearly grow if it is to remain embedded in Euclidean space, but the corresponding body-diagonals remain unchanged. This produces a growing deficit angle around the body-diagonals, shown on the right of figure \ref{fig:BodyDiag}, which then drives the growth of the face-diagonals. The addition of the body diagonal to each block can also be interpreted as producing an over-determined system, with seven edge-lengths associated with each vertex or block, while there are only six metric components at each point of a smooth manifold. However, this interpretation also suggests a solution.

The flat segments of a piecewise flat manifold do not necessarily have to be tetrahedra, these are just the most simple of segments. If each block of the cubic and skew triangulations are instead treated as flat, the mechanism that drives the exponential growth of the face-diagonals will be broken. This has lead to the following:

\begin{proposition}
	\label{prop:stab}
	The exponential instability is suppressed for the piecewise flat Ricci flow of cubic and skew type piecewise flat manifolds with flat blocks as the piecewise flat segments.
\end{proposition}

In practice the body-diagonals are retained since it is easier to compute dihedral angles and volumes with a tetrahedral graph. Their lengths are then continually re-defined to give zero deficit angles around them, and hence a flat interior for each block, as shown in figure \ref{fig:stable}. 
This results in a set of constraint equations to determine the lengths of the body-diagonals at each step of an evolution, circumventing the over-determined nature of the tetrahedral triangulations.

\begin{figure}[h]
	\centering
	\includegraphics[scale=1.2]{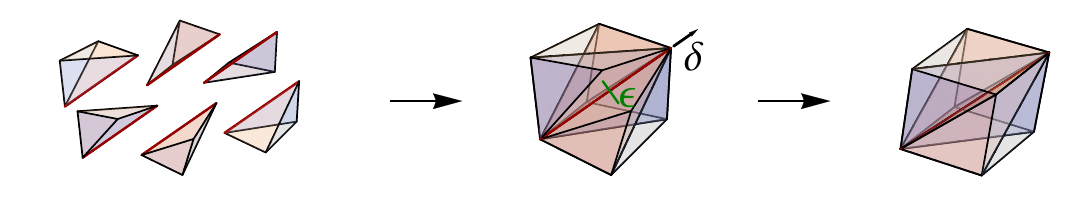}
	\caption{The deficit angle $\epsilon$ at the body-diagonal is shown, along with the perturbation $\delta$ of the body-diagonal that makes this deficit angle zero, giving a flat interior for the block.}
	\label{fig:stable}
\end{figure}

Proposition \ref{prop:stab} is proved for a number of cubic and skew grids of $T^3$ manifolds in section \ref{sec:ProofStab}, and numerical simulations show the suppression of the instability in section \ref{sec:NumSimStab}. The use of flat blocks has also given stable evolutions for all of the computations in \cite{PLRF}, giving remarkably close approximations to their corresponding smooth Ricci flows.

\section{Initial instability}
\label{sec:Instability}

Since it is the linear terms in a set of differential equations that lead to exponential growth, the linear stability of the piecewise flat Ricci flow was tested for cubic and skew triangulations of flat $T^3$ manifolds.

\subsection{Linear stability analysis}
\label{sec:LinStab}

A linear stability analysis uses the linear terms of a perturbation away from an equilibrium to test for the stability of that equilibrium.

\begin{definition} For a system of differential equations $\frac{d x_i}{d t} = f_i (x_j)$:
	\label{def:LinStabTerms}
	\begin{enumerate}
		\item A stationary solution $x_i = x_i^0$ is a solution that does not change with $t$, i.e. $\frac{d x_i^0}{d t} = 0$.
		
		\item Linearized equations at $x^0_i$ are the linear terms in a series expansion of $f_i(x_j^0 + \delta_j)$ about $\delta_j = 0$. The zero order terms vanish since they correspond to a stationary solution, resulting in the equations
		$\frac{d \delta_i}{d t} 
		= a_{i j} \, \delta_j$
		with real numbers $a_{i j}$.
		
		\item The system is linearly unstable at $x^0_i$ if the coefficient matrix $A = a_{i j}$ for the linearized equations has any eigenvalues with positive real parts. Solutions of the linearized equations consist of linear combinations of exponential functions, with the eigenvalues of $A$ giving the growth rates.
	\end{enumerate}
\end{definition}

Euclidean metrics provide stationary solutions for the smooth Ricci flow, having zero Ricci curvature, and triangulations of flat Euclidean manifolds are stationary for the piecewise flat Ricci flow, with zero deficit angles and therefore zero piecewise flat Ricci curvature. The edge-lengths for cubic and skew triangulations of flat Euclidean space, with unit volume blocks, are given in table \ref{tab:flatLengths} below.

\begin{table}[h!]
	\centering
	\begin{tabular}{l|ccc|ccc|c}
		& $\ell_x$ & $\ell_y$ & $\ell_z$ 
		& $\ell_{y z}$ & $\ell_{z x}$ & $\ell_{x y}$ 
		& $\ell_{x y z}$
		\\
		\hline
		Cubic
		& $1$ & $1$ & $1$
		& $\sqrt{2}$ & $\sqrt{2}$ & $\sqrt{2}$
		& $\sqrt{3}$
		\\
		Skew
		& $\frac{1}{3}\sqrt{10}$ & $1$ & $\frac{1}{9}\sqrt{94}$
		& $\frac{1}{9}\sqrt{139}$ & $\frac{1}{9}\sqrt{142}$ & $\frac{1}{3}\sqrt{13}$
		& $\frac{1}{9}\sqrt{133}$
		\\
	\end{tabular}
	\caption{The edge-lengths for flat triangulations with unit volume blocks.}
	\label{tab:flatLengths}
\end{table}

Any global scaling of the edge-lengths in table \ref{tab:flatLengths} will also be stationary solutions of the piecewise flat Ricci flow. However, the linearized equations and the eigenvalues of the coefficient matrix are not invariant to this rescaling. The effect of globally rescaling the triangulation blocks is therefore given below.

\begin{lemma}[Scale factor]
	\label{lem:scale}
	If a triangulation of a flat manifold has a coefficient matrix $A$, then the coefficient matrix for a rescaling of all of the edges by a factor of $c$ will be $\frac{1}{c^2} A$.
\end{lemma}

\begin{proof}
	From (\ref{eq:RandK}), the piecewise flat Ricci curvature for an edge $\ell_i$ can be written as the sum
	\begin{equation}
	\label{eq:scale1}
	Rc_i 
	= \sum_k b_{i k} \frac{|\ell_k| \, \epsilon_k}{V_k},
	\end{equation}
	for some coefficients $b_{i k}$. The series expansion of $Rc_i$ for a perturbation $\delta_j$ of some edge $\ell_j$ is given by the series expansions of the individual terms appearing on the right-hand side above. The zero-order terms for the deficit angles $\epsilon_k$ will always be zero, since these correspond to a triangulation of Euclidean space. Hence, the linear terms in the expansion of $Rc_i$ must be given by the linear terms from the deficit angles, and the zero-order terms, or non-perturbed values, for the remaining variables.
	
	For a global rescaling of all the edge-lengths of a triangulation by a factor of $c$, the volumes are clearly scaled by $c^3$. The coefficients $b_{i k}$ can be seen in (\ref{eq:RandK}) to be either constant or depend on the angles between edges, and therefore not depend on the scaling. This gives the relations
	\begin{equation}
	|\ell_k^c| = c \,|\ell_k|
	, \qquad
	V_k^c = c^3 \, V_k
	, \qquad
	b_{i k}^c = b_{i k},
	\end{equation}
	with the superscript $c$ representing the rescaled terms. The deficit angles depend on the relative lengths of the edges, and since the perturbation $\delta_j$ is the only length that is \emph{not} rescaled by $c$, the deficit angle would be the same as if only $\delta_j$ was rescaled, but by a factor of $1/c$. An expansion of the deficit angle $\epsilon_k^c(\delta_j)$ for the rescaled blocks is therefore given by the equation
	\begin{equation}
	\epsilon_k^c (\delta_j) \
	= \ \epsilon_{k j}^c \, \delta_j + O(\delta_j^2) \
	= \ \frac{1}{c} \, \epsilon_{k j} \, \delta_j + O(\delta_j^2) \ ,
	\end{equation}
	with $\epsilon_{k j}^c$ and $\epsilon_{k j}$ representing the first order coefficients.
	
	Using the piecewise flat Ricci flow equation (\ref{eq:PFRF}), the linear coefficients $a_{i j}^c$ for the rescaled triangulation can now be given in terms of the linear coefficient $a_{i j}$ for the original triangulation,
	\begin{equation}
	a_{i j}^c \
	= \ - |\ell_i^c| \, \sum_k b_{i k}^c 
	\frac{|\ell_k^c| \, \epsilon_{k j}^c}{V_k^c} \
	= \ - \left(c \, |\ell_i|\right) \sum_k b_{i k} 
	\frac{\left(c \, |\ell_k|\right) \, 
		\left(\frac{1}{c}\epsilon_{k j}\right)}{c^3 \, V_k} \
	= \ \frac{1}{c^2} \, a_{i j}.
	\end{equation}
	The coefficient matrix $A$, and hence its eigenvalues, are therefore scaled by a factor of $1/c^2$ when all the edge-lengths of a triangulation are rescaled by $c$.
\end{proof}

\subsection{Proof of proposition \ref{prop:instab}}
\label{sec:ProofInstab}

To calculate the linearized equations for the piecewise flat Ricci flow, a number of properties of both the graph structure and the linearization itself are taken advantage of.

\begin{itemize}
    \item It is only necessary to determine the linearized equations for a single set of three face-diagonals due to the symmetry of the grids. The equations for all of the other face-diagonals will have the same coefficients, with an appropriate translation of indices.
    
    \item A $3 \times 3 \times 3$ grid of cubic or skew blocks provides all of the edges required to determine the piecewise flat Ricci curvature for the edges in the central block. This can be seen from (\ref{eq:RandK}), where the piecewise flat Ricci curvature $Rc_\ell$ depends only on the deficit angles at edges $\ell_j$ that have a vertex in common with $\ell$, and these depend only on the lengths of edges in tetrahedra containing the edge $\ell_i$.
    
    \item Series expansions of the piecewise flat Ricci curvature need only be computed for a single perturbation variable at a time, with the linear terms for each perturbation computed separately and then added together to give the complete linearized equation. This avoids the need to compute series expansions of multiple variables simultaneously.
\end{itemize}
Symbolic manipulations in \emph{Mathematica} were used to calculate the linearized equations for both the cubic and skew triangulations, with the code and results available in the Zenodo repository at https://doi.org/10.5281/zenodo.8067524.

\begin{theorem}[Linear instability of cubic triangulations]
\label{lem:CubicInstab}
	The piecewise flat Ricci flow of any cubic triangulation of a flat $T^3$ manifold is linearly unstable, with perturbations growing exponentially at a rate of at least $12/c^2$ for cubic blocks with volume $c^3$.
\end{theorem}

\begin{proof}
	To begin, the linearized equations for the piecewise flat Ricci flow equations about a flat cubic triangulation with unit volume blocks are calculated. Each face-diagonal in a $3 \times 3 \times 3$ grid of cubic blocks is perturbed in turn, with the linearized equations at a set of three face-diagonals in the central block computed for each perturbation, and the contributions from all of these perturbation then added together. The resulting linearized equation at the $x y$-face-diagonal takes the form:
	\begin{align}
	\label{eq:lxyCubic}
	\frac{d}{d t} \, \delta_{x y} &(0,0,0)
	=
	\nonumber \\ & 
	- 4 \ \delta_{x y} (0,0,0)
	-     \delta_{x y} (-1,-1,0)
	-     \delta_{x y} (1,1,0)
	\nonumber \\ &
	+ \frac{3}{2} \ \big(
		  \delta_{x y} (-1,0,0)
		+ \delta_{x y} (0,-1,0)
		+ \delta_{x y} (0,1,0)
		+ \delta_{x y} (1,0,0)
		\big)
	\nonumber \\ &
	+ \, 2 \ \big(
		  \delta_{x y} (0,0,-1)
		+ \delta_{x y} (0,0,1)
		\big)
	\nonumber \\ &
	\nonumber \\ &
	- \frac{1}{2} \ \big(
		  \delta_{y z} (0,-1,-1)
		+ \delta_{y z} (1,1,0)
		\big)
		+ \delta_{y z} (0,-1,0)
		+ \delta_{y z} (1,1,-1)
	\nonumber \\ &
	+ \frac{3}{2} \ \big(
		  \delta_{y z} (0,0,0)
		+ \delta_{y z} (1,0,-1)
		\big)
	\nonumber \\ &
	\nonumber \\ &
	- \frac{1}{2} \ \big(
		  \delta_{z x} (-1,0,-1)
		+ \delta_{z x} (1,1,0)
		\big)
		+ \delta_{z x} (-1,0,0)
		+ \delta_{z x} (1,1,-1)
	\nonumber \\ &
	+ \frac{3}{2} \ \big(
		  \delta_{z x} (0,0,0)
		+ \delta_{z x} (0,1,-1)
		\big)
		,
	\end{align}
	with the coordinates in parentheses indicating the location of the perturbed edge in the triangulation grid, with respect to the central block. Due to the symmetries in the cubic lattice, the linearized equations for the other two face-diagonals are given by a permutation of the $\{x y, y z, z x\}$ subscripts and a similar permutation of the the grid coordinates. The linearized equation for \emph{any} face-diagonal in a $T^3$ grid of cubic blocks can then be given by a discrete linear transformation of the grid coordinates.

	The set of coefficients in (\ref{eq:lxyCubic}) are the same for the linearized equations at all of the face-diagonals in the triangulation, for any size grid, so the set of elements in each row of the coefficient matrix $A$ will also be the same. These elements sum to $12$, which must be an eigenvalue of $A$ with a corresponding eigenvector consisting of all ones. From lemma \ref{lem:scale}, the coefficient matrix for a triangulation with blocks of volume $c^3$ will have an eigenvalue of $12/c^2$. Any solution to the set of linearized equations must then contain an exponential term with a growth rate of $12/c^2$, leading to an exponential growth for the perturbations in all of the face-diagonals of at least this rate.
\end{proof}

\begin{theorem}[Linear instability of skew triangulations]
\label{lem:SkewInstab}
	The piecewise flat Ricci flow of any skew triangulation of a flat $T^3$ manifold is linearly unstable, with perturbations growing exponentially at a rate of at least $0.996/c^2$ for cubic blocks with volume $c^3$.
\end{theorem}

\begin{proof}
	As with the cubic triangulations in lemma \ref{lem:CubicInstab}, the linearized equations about a flat skew triangulation with unit volume blocks are first computed. Unlike the cubic case, the skew blocks do not have the same symmetries as the cubic blocks, so the linearized equations for each of the three types of face-diagonals, $\ell_{x y}$, $\ell_{y z}$ and $\ell_{z x}$ must be found separately. The linearized equations are not displayed here, but can be found in the Zenodo repository, https://doi.org/10.5281/zenodo.8067524. 
	As with the cubic case, the linearized equations for all of the face-diagonals in a $T^3$ grid of skew blocks can be given by a discrete transformation of the grid coordinates for each of the three face-diagonals in a single block.

	From these equations, it can be seen that the sum of the coefficients are not the same for each face-diagonal, so the vector of all ones is not an eigenvector for the coefficient matrix $A$ as it was for the cubic triangulations. However, by ordering the indices of the face-diagonals $\ell_i$ according to their edge-type, a similar approach can be used. The indices for an $n$-block triangulation are defined so that $i \in \{1, ..., n\}$ for the $yz$-diagonals, $i \in \{n+1, ..., 2n\}$ for the $zx$-diagonals and $i \in \{2n + 1, ..., 3n\}$ for the $xy$-diagonals. Defining a vector $v$ so that
	\begin{equation}
	\label{eq:SkewVec}
	v = 
	\left\{
	\begin{array}{ccc}
		p & \textrm{if} & 1 \leq i \leq n \\
		q & \textrm{if} & n+1 \leq i \leq 2n \\
		r & \textrm{if} & 2n+1 \leq i \leq 3n \\
	\end{array}
	\right. \quad \textrm{for } p, q, r, \in \mathbb{R},
	\end{equation}
	the product of the matrix $A$ with $v$ is
	\begin{equation}
	\label{eq:SkewSum}
	A \ v 
	 = a_{i j} \cdot v_j 
	 = \left(\sum_{j = 1}^{n} a_{i j}\right) p 
	 + \left(\sum_{j = n+1}^{2n} a_{i j}\right) q 
	 + \left(\sum_{j = 2n+1}^{3n} a_{i j}\right) r .
	\end{equation}
	Since there will only be three different values for the elements of the resulting vector, one for each type of face-diagonal, the information in this product can be reduced to the $3 \times 3$ matrix product below,
	\begin{equation}
	\label{eq:SkewInstab3mat}
	\left(
	\begin{array}{ccc}
		0.308 & 0.311 & 0.282 \\
		0.410 & 0.415 & 0.376 \\
		0.266 & 0.269 & 0.244 \\
	\end{array}
	\right) \left(
	\begin{array}{c}
		p \\
		q \\
		r \\
	\end{array}
	\right) ,
	\end{equation}
	with the matrix elements obtained by summing the appropriate coefficients in the linearized equations. This matrix has a maximum eigenvalue of approximately $0.966$, with a corresponding eigenvector of $(0.532, 0.710, 0.461)$. From (\ref{eq:SkewSum}), the matrix $A$ must also have this eigenvalue, with eigenvector $v$ from (\ref{eq:SkewVec}) where $p = 0.532$, $q = 0.710$ and $r = 0.461$.

	From lemma \ref{lem:scale}, the coefficient matrix for a triangulation with blocks of volume $c^3$ will have an eigenvalue of approximately $0.966/c^2$. Any solution to the set of linearized equations must then contain an exponential term with a growth rate of approximately $0.966/c^2$, leading to an exponential growth for the perturbations in all of the face-diagonals of at least this rate.
\end{proof}

The linearized equations in the proofs of lemmas 5 and 6 have also been used to construct the coefficient matrices for $3 \times 3 \times 3$ and $3 \times 3 \times 4$ grids of blocks using \emph{Mathematica}. This was done for both the cubic and skew blocks, with the edge-lengths from table \ref{tab:flatLengths}, and for a rescaling of these edges by a factor of $1/3$. The eigenvalues for each matrix were then computed, again using \emph{Mathematica}, with the largest real parts matching the  eigenvalues in lemmas \ref{lem:CubicInstab} and \ref{lem:SkewInstab}, as shown in table \ref{tab:instabEvals}.

\begin{table}[h!]
	\centering
	\begin{tabular}{l|cc|cc|}
	    \multicolumn{1}{c}{}
	    &\multicolumn{2}{c}{$c = 1$}
	     &\multicolumn{2}{c}{$c = 1/3$}
	    \\
		&  $3 \times 3 \times 3$
         & $3 \times 3 \times 4$
         & $3 \times 3 \times 3$
         & $3 \times 3 \times 4$
		\\
		\hline
		Cubic
		& $ 12, \,  6, \, 2.739,  \,  0$ 
		& $ 12, \,  8, \, 6,   \, 4.514$
		& $108, \, 54, \, 24.65,   \, 0$ 
		& $108, \, 72, \, 54,  \, 40.63$  
		\\
		Skew
		& $0.966, \ 0.$ 
		& $0.966, \ 0.$ 
		& $8.697, \ 0.$ 
		& $8.697, \ 0.$ 
		\\
	\end{tabular}
	\caption{The largest of the real parts of the eigenvalues for both cubic and skew triangulations, with two different grid sizes and two different scales $c$. The non-integer values are approximated to three decimal places.
	}
	\label{tab:instabEvals}
\end{table}

\begin{remark}
Due to the effect of the scaling factor $c$ in lemma 4, the instabilities are more severe when the grid resolutions are increased. This is the opposite of the piecewise flat approximations, which should converge to their corresponding smooth values as the resolution is increased.
\end{remark}

\

\begin{remark}
The instability for the skew triangulations is an order of magnitude less than for the cubic triangulations for the same block volumes. It can also be noted that the cubic blocks are only borderline Delaunay for a flat manifold, with the circumcentres of all tetrahedra in a single block coinciding at the centre of that block. The skew blocks were initially used because they form more strongly Delauney triangulations, where Voronoi dual volumes can be used with more confidence. For a flat diamond block, the circumcentres of all of the tetrahedra coincide with their barycentres, making them as strongly Delaunay as possible, and may offer a clue to explain the original lack of instability in the diamond triangulations.
\end{remark}

\subsection{Numerical simulations}
\label{sec:NumSim}

Simulations have been run for the piecewise flat Ricci flow of $3\times3\times3$, $4 \times 4 \times 4$ and $5 \times 5 \times 5$ grids of both cubic and skew blocks. The base edge-lengths were taken from table \ref{tab:flatLengths}, scaled by a factor of $1/3$ for the skew triangulations, with lemmas \ref{lem:CubicInstab} and \ref{lem:SkewInstab} indicating that these should give exponential growth rates on the order of $10$. The edge-lengths were then approximated as double-precision floating point numbers, and each perturbed by a random number from a normal distribution with standard deviation of $10^{-15}$, the level of numerical precision. Evolutions were performed using an Euler method with 100 steps of size 0.01, and deviations in the face-diagonal edge-lengths were fitted to a linear combination of exponential functions, 
\begin{align}
f_{cubic}(t) &= a_1 e^{k_1 t} + a_2 e^{k_2 t} + a_3 e^{k_3 t} + c \, ,
\nonumber \\
f_{skew}(t) &= a_1 e^{k_1 t}
+ b \, t + c \,.
\label{eq:best-fit}
\end{align}
The number of terms was chosen to include all of the positive eigenvalues for the $3 \times 3 \times 3$ grid triangulations shown in table \ref{tab:instabEvals}. A linear term was also added for the skew function, as the non-face-diagonal edges showed a consistent linear growth of about $10^{-15}$. Results of the growth rates and $R$-squared values for the best-fit functions are presented in table \ref{tab:pertSim}, with sample graphs of the fitted functions and their corresponding data shown in figure \ref{fig:NumFit}.

\begin{table}[h!]
	\centering
	\begin{tabular}{llc|cc|cc|cc}
	    \multicolumn{3}{c}{}
		&\multicolumn{2}{c}{$3 \times 3 \times 3$}
		&\multicolumn{2}{c}{$4 \times 4 \times 4$}
		&\multicolumn{2}{c}{$5 \times 5 \times 5$}
		\\
		&& Lin. App. 
		& Median & IQR 
		& Median & IQR 
		& Median & IQR
		\\
		\cline{2-9}
		Cubic \ 
		& $k_1$    & $12$ 
		& $11.998$ & $0.001$ 
		& $11.997$ & $0.007$ 
		& $11.97$  & $0.04$ 
		\\
		& $k_2$   & $6$ 
		& $6.04$ & $0.01$ 
		& $6.04$ & $0.08$ 
		& $6.2$  & $0.5$ 
		\\
		& $k_3$  & $2.739$ 
		& $2.86$ & $0.04$ 
		& $2.9$  & $0.2$ 
		& $3.8$  & $1.9$ 
		\\
		& $R^2$  &  
		& $0.99998$ & $10^{-6}$ 
		& $0.99994$ & $10^{-4}$ 
		& $0.99930$ & $10^{-3}$ 
		\\
		&&&&&&&& \\
		Skew \ 
		& $k$ & $8.697$ 
		& $8.339$ & $0.004$ 
		& $8.336$ & $0.01$ 
		& $8.337$ & $0.01$ 
		\\
		& $R^2$  &  
		& $0.999999$ & $10^{-7}$ 
		& $0.999998$ & $10^{-6}$ 
		& $0.999999$ & $10^{-6}$ 
		\\
	\end{tabular}
	\caption{
		The median growth rates and $R$-squared values with their interquartile ranges (IQR) for each triangulation. The values for the cubic triangulations are in extremely close agreement with the corresponding eigenvalues for the $3 \times 3 \times 3$ grid in table \ref{tab:instabEvals}, and the skew parameters in close agreement.
	}
	\label{tab:pertSim}
\end{table}

\begin{figure}[h!]
	\begin{center}
		\includegraphics[scale=1]{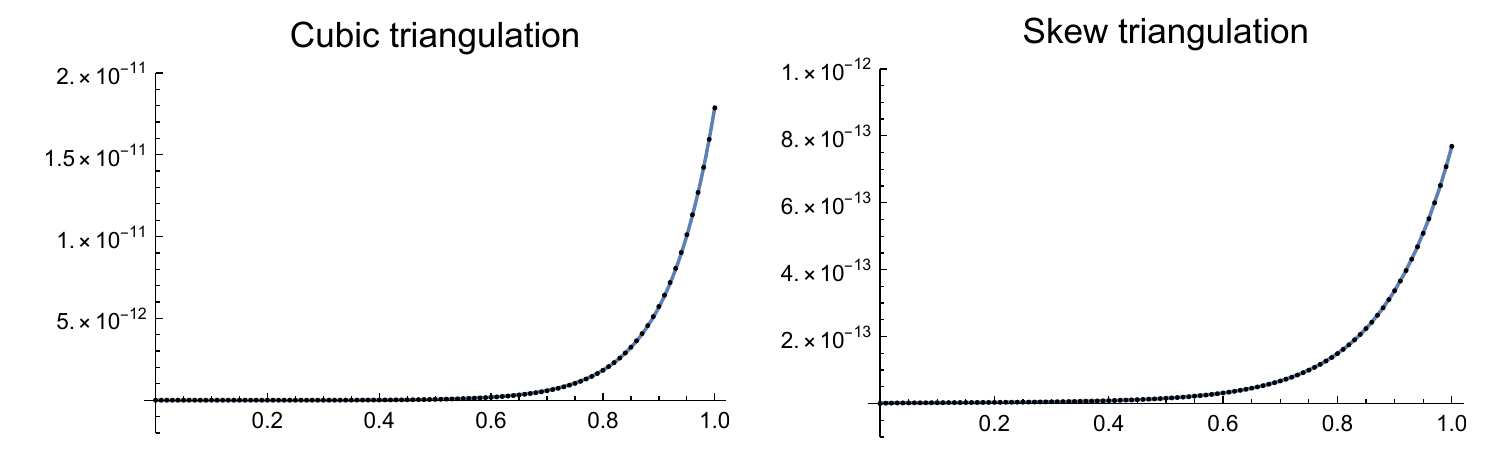}
	\end{center}
	\vspace{-0.5cm}
	\caption{
		The best-fit graphs with the \emph{lowest} $R$-squared values ($0.99998$ and $0.999999$ respectively) for deviations in the edge-lengths of the $3 \times 3 \times 3$ grid triangulations, closely agreeing with the evolution values represented by the points.
	}
	\label{fig:NumFit}
\end{figure}

The close agreement between the growth rates in table \ref{tab:pertSim} and eigenvalues in table \ref{tab:instabEvals}, along with the extremely high $R$-squared values, demonstrate how effective the linearized equations are in approximating the evolution. The low interquartile ranges also show consistent behaviour across all edges, with surprisingly comparable growth rates over the different grid sizes, considering a higher number of distinct eigenvalues should be expected for larger grids. In particular, the simulations support the hypothesis that the eigenvalues in lemmas \ref{lem:CubicInstab} and \ref{lem:SkewInstab} represent the largest growth rates for any cubic or skew type triangulations.

\section{Suppression of Exponential Instability}
\label{sec:Stability}

With the body-diagonals re-defined to give flat interiors for each block, the linear stability analysis and numerical simulations are performed again, with the exponential growth suppressed in both.

\subsection{
Linear instability suppressed
}
\label{sec:ProofStab}

The linearized equations are calculated here by taking advantage of the same properties outlined at the beginning of section \ref{sec:ProofInstab}, but with some additional steps to ensure that each block is flat after any perturbations.

\begin{itemize}
	\item Once a face-diagonal $\ell_j$ is perturbed by an arbitrary amount $\delta_j$, only the blocks on either side of that face are affected. 
	For the body-diagonals of each of these blocks, the deficit angle can be found in terms of $\delta_j$ and the length $b$ of the body-diagonal itself.

	\item Setting the deficit angle to zero, the length $b$ which gives a flat block can be found in terms of $\delta_j$. Since only linear terms in $\delta_j$ will impact the linearized equations, $b$ need only be found as a linear approximation in $\delta_j$.
	
	\item A grid of size $4 \times 4 \times 4$ is now required to determine the piecewise flat Ricci flow for the edges in a central block, due to the changes in body-diagonals.
\end{itemize}

The linearized equations are first used to show that the approaches in lemmas \ref{lem:CubicInstab} and \ref{lem:SkewInstab} no longer imply a linear instability once flat blocks are used, and then to show that the linear instability is actually suppressed for a number of different grid sizes.

\begin{theorem}
\label{lem:ZeroEvalue}
	When the body-diagonals are re-defined to give flat blocks, summing rows of the linear coefficient matrix 
	no longer implies a linear instability 
	of the piecewise flat Ricci flow.
\end{theorem}

\begin{proof}
	Each face-diagonal $\ell_j$ in a $4 \times 4 \times 4$ grid of blocks is perturbed away from the flat values in table \ref{tab:flatLengths} by an arbitrary amount $\delta_j$, with the body-diagonals on either side of that face re-defined in terms of $\delta_j$ to give a zero deficit angle. The linear impact of each perturbation $\delta_j$ on a face-diagonal $\ell_i$ in a central block can then be calculated using the piecewise flat Ricci flow equations (\ref{eq:PFRF}), and the separate terms summed to give the linearized equation for $\ell_i$.
	
	The resulting equation is shown below for an $xy$-face-diagonal of the cubic triangulation, 
	with the coordinates indicating the relative locations of the face-diagonals on the grid,
	\begin{align}
	\frac{d}{d t} \, \delta_{x y} &(0,0,0) =
	\nonumber \\ & 
	- 5 \ \delta_{x y} (0,0,0)
	- \delta_{x y} (-1,-1,0)
	- \delta_{x y} (1,1,0)
	\nonumber \\ &
	- 1/4 \big(
		  \delta_{x y} (-1,0,1)
		+ \delta_{x y} (0,-1,1)
		+ \delta_{x y} (0,1,-1)
		+ \delta_{x y} (1,0,-1)
		\big)
	\nonumber \\ &
	+ 5/4 \big(
		  \delta_{x y} (-1,0,0)
		+ \delta_{x y} (0,-1,0)
		+ \delta_{x y} (0,1,0)
		+ \delta_{x y} (1,0,0)
		\big)
	\nonumber \\ &
	+ 3/2 \big(
		  \delta_{x y} (0,0,-1)
		+ \delta_{x y} (0,0,1)
		\big)
	\nonumber \\ &
	\nonumber \\ &
	- 1/2 \big(
		\delta_{y z} (0,-1,-1)
		+ \delta_{y z} (0,0,-1)
		+ \delta_{y z} (1,0,0)
		+ \delta_{y z} (1,1,0)
		\big)
	\nonumber \\ &
	- 1/4 \big(
		\delta_{y z} (-1,0,0)
		+ \delta_{y z} (0,1,-1)
		+ \delta_{y z} (1,-1,0)
		+ \delta_{y z} (2,0,-1)
		\big)
	\nonumber \\ &
	+ 3/4 \big(
		\delta_{y z} (0,-1,0)
		+ \delta_{y z} (0,0,0)
		+ \delta_{y z} (1,0,-1)
		+ \delta_{y z} (1,1,-1)
		\big)
	\nonumber \\ &
	\nonumber \\ &
	- 1/2 \big(
		  \delta_{z x} (-1,0,-1)
		+ \delta_{z x} (0,0,-1)
		+ \delta_{z x} (0,1,0)
		+ \delta_{z x} (1,1,0)
		\big)
	\nonumber \\ &
	- 1/4 \big(
		  \delta_{z x} (-1,1,0)
		+ \delta_{z x} (0,-1,0)
		+ \delta_{z x} (0,2,-1)
		+ \delta_{z x} (1,0,-1)
		\big)
	\nonumber \\ &
	+ 3/4 \big(
		  \delta_{z x} (-1,0,0)
		+ \delta_{z x} (0,0,0)
		+ \delta_{z x} (0,1,-1)
		+ \delta_{z x} (1,1,-1)
		\big)
	.
	\label{eq:lxyF}
	\end{align}
	As in the proof of lemma \ref{lem:CubicInstab}, the linearized equations for the other two types of face-diagonal in the cubic triangulation are simply given by a permutation of the $\{x y, y z, z x\}$ subscripts and of the grid coordinates.  The linearized equations for a skew triangulation are not displayed here, but can be found in the Zenodo repository at https://doi.org/10.5281/zenodo.8067524. 
	The linearized equations for all of the face-diagonals in any $T^3$ grid of cubic or skew blocks can be found through a discrete transformation of the grid coordinates.
	
	The coefficients in equation (\ref{eq:lxyF}) can be seen to sum to zero, with the coefficients of the linearized equations for the face-diagonals in the skew triangulation also summing to zero. This implies that the vector of all ones is an eigenvector of the coefficient matrix $A$, with an eigenvalue of zero, for all grid-sizes of both the cubic and skew triangulations. While this approach gave positive eigenvalues in lemmas \ref{lem:CubicInstab} and \ref{lem:SkewInstab}, proving the equations there to be linearly unstable, it does not do so when flat blocks are used instead of just flat tetrahedra.
\end{proof}

\begin{remark}
For many constant row-sum matrices, the eigenvalue given by the row-sum can be proved to be the largest eigenvalue. For example, if all of the coefficients except the first in equation \ref{eq:lxyF} were positive, the Gershgorin circle theorem would be sufficient to prove this. Recent progress has also been made on eigenvalue bounds for more general constant row-sum matrices \cite{RowSumMH19,RowSumBS23}. Unfortunately, a proof that the same is true in this case has not been found, but direct computation of the eigenvalues below and the numerical simulations in section \ref{sec:NumSimStab} suggest that it is true.
\end{remark}

\

For particular grid sizes, the linearized equations were used to construct the coefficient matrices directly, using the mathematical software \emph{Mathematica}, and the set of eigenvalues computed for each. The eigenvalue with the largest real part is given for each piecewise flat manifold in table \ref{tab:stabEvals}. For the cubic blocks these were computed exactly, but numerical methods were required to compute the eigenvalues for the skew blocks.

\begin{table}[h!]
	\centering
	\begin{tabular}{l|cc|cc|}
		\multicolumn{1}{c}{}
		&\multicolumn{2}{c}{$c = 1$}
		&\multicolumn{2}{c}{$c = 1/3$}
		\\
		& $3 \times 3 \times 3$
		& $3 \times 3 \times 4$
		& $3 \times 3 \times 3$
		& $3 \times 3 \times 4$
		\\
		\hline
		  Cubic
		& $0$ 
		& $0$ 
		& $0$ 
		& $0$
		\\
		  Skew
		& $3.0 \times 10^{-15}$ 
		& $3.3 \times 10^{-15} $ 
		& $2.0 \times 10^{-14}$ 
		& $3.7 \times 10^{-14} $ 
		\\
	\end{tabular}
	\caption{The eigenvalues with the largest real part, showing a suppression of the linear instability when flat blocks are used.
	}
	\label{tab:stabEvals}
\end{table}

Clearly, there cannot be any exponential growth terms for the linearized equations in the cubic triangulations since the largest real part of the eigenvalues is zero. The largest values for the skew triangulation are \emph{practically} zero, being equivalent to zero for the numerical precision of the computations. While an eigenvalue of zero does not imply linear stability, it is consistent with the linearization of the smooth Ricci flow near a flat manifold, which also contains zero eigenvalues \cite{GIKstability}. Also, since the piecewise flat curvatures depend on local regions of a piecewise flat manifold, it is deemed unlikely that the instability would re-emerge in larger grids.

\subsection{Numerical simulations with instability suppressed}
\label{sec:NumSimStab}

Numerical simulations of the piecewise flat Ricci flow have also been run with blocks that are effectively flat. In order to compare with the simulations in section \ref{sec:NumSim}, the same triangulations and edge-length perturbations were used, and evolved using the same Euler method with 100 steps of size 0.01, but with the body-diagonal edge-lengths adapted at the beginning of each step. This was done by first adding a perturbation variable $\delta_j$ to the length of each body-diagonal $\ell_j$, computing the deficit angle around $\ell_j$ as a function of $\delta_j$, and then setting this equal to zero and solving for $\delta_j$. Since the deficit angles should already be small for a good triangulation, a linear approximation in $\delta_j$ about zero is used for the deficit angle, giving a unique solution for $\delta_j$ when set equal to zero. To show that the exponential growth is suppressed, the median and maximum changes in edge-lengths for all triangulations are shown in table \ref{tab:SuppChanges}, with the deviations of the edge-lengths for the blocks with the largest changes shown in figure \ref{fig:NumNonExp}. These same edges were used for the graphs in figure \ref{fig:NumErr}.

\begin{table}[h!]
	\centering
	\centering
	\begin{tabular}{ll|ccc}
		&
		&$3 \times 3 \times 3$
		&$4 \times 4 \times 4$
		&$5 \times 5 \times 5$
		\\
		\hline
		Cubic
		&Median change \
		& $0$ 
		& $0$ 
		& $0$ 
		\\
		&Max. change \
		& $1.7 \times 10^{-15}$ 
		& $3.1 \times 10^{-15}$ 
		& $2.2 \times 10^{-15}$ 
		\\
		&Initial pert. \
		& $2.4 \times 10^{-15}$ 
		& $3.1 \times 10^{-15}$ 
		& $3.3 \times 10^{-15}$ 
		\\
		&&&& \\
		Skew
		&Median change \
		& $4.6 \times 10^{-15}$ 
		& $4.6 \times 10^{-15}$ 
		& $4.6 \times 10^{-15}$ 
		\\
		&Max. change \
		& $6.7 \times 10^{-15}$ 
		& $8.0 \times 10^{-15}$ 
		& $8.1 \times 10^{-15}$ 
		\\
		&Initial pert. \
		& $2.9 \times 10^{-15}$ 
		& $3.1 \times 10^{-15}$ 
		& $4.0 \times 10^{-15}$ 
		\\
		&&&& \\
		&Median slope \
		& $4.5 \times 10^{-15}$ 
		& $4.6 \times 10^{-15}$ 
		& $4.5 \times 10^{-15}$ 
		\\
		&IQR of slope \
		& $7.5 \times 10^{-16}$ 
		& $7.8 \times 10^{-16}$ 
		& $7.6 \times 10^{-16}$ 
		\\
	\end{tabular}
	\caption{
		The median and maximum edge-length changes for each triangulation, along with the maximum values of the initial random perturbations. The median and interquartile ranges (IQR) for the best-fit linear functions in the skew triangulations are also shown. All values are close to the numerical precision, showing a suppression of the exponential growth.
	}
	\label{tab:SuppChanges}
\end{table}

The adapting of the body-diagonals has clearly suppressed the exponential growth for both the cubic and skew simulations. For the cubic triangulations, most of the edge-lengths do not change, and those that do, only change during the first quarter of time steps and by an extremely small amount, less than the largest of the initial perturbations. While the growth rates are non-zero throughout the simulation, as seen in figure \ref{fig:NumNonExp}, when combined with the step size for the Euler method the resulting changes are lower than the numerical precision. The set of edge-lengths then become stationary once the rates of change drop below this threshold.

For the skew triangulations, the numerical precision is lower as it depends on the lengths of the edges, so the oscillations in the lower-right graph in figure \ref{fig:NumNonExp} continue to produce a linear growth. Linear functions have therefore been fitted to the data for each edge in the triangulation, with the results in table \ref{tab:SuppChanges} showing consistent, extremely small rates of change across all edges, also agreeing with the rates of the background linear growth for the un-suppressed simulations in section \ref{sec:NumSim}.
While this linear growth does not directly correspond to smooth Ricci flow, it does not conflict with it either. The consistency of the rates of change lead to a global change in the scale but will not produce any curvature, so the piecewise flat manifold remains in a stationary state, growing but remaining flat. The low rate means the effect will not be noticeable at regular scales for extremely long time frames, but the effect can still be avoided by using the normalized Ricci flow which preserves the global volume.

\begin{figure}[h!]
	\begin{center}
		\includegraphics[scale=1]{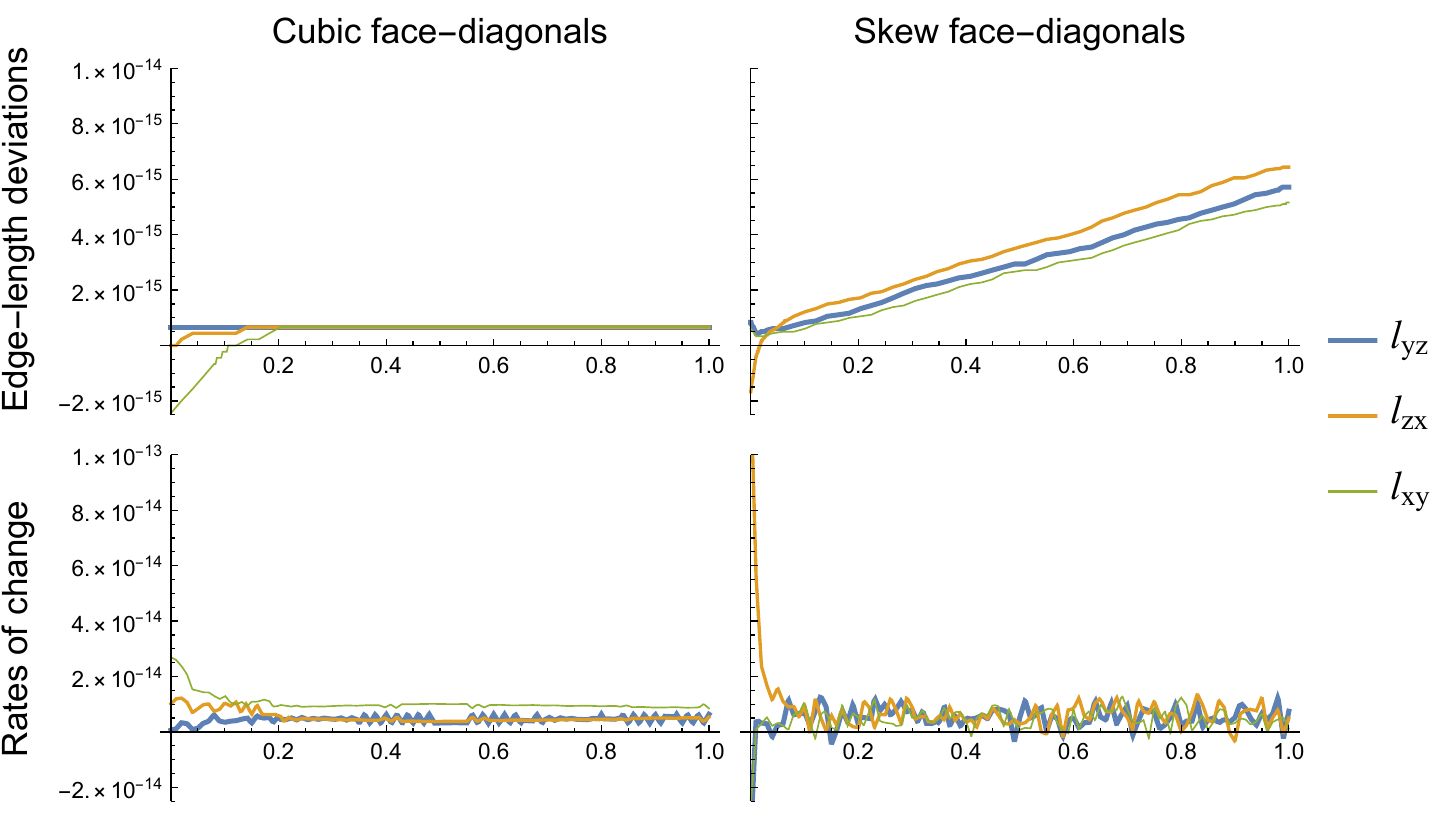}
	\end{center}
	\vspace{-0.5cm}
	\caption{
		Deviations of the face-diagonal edge-lengths from their pre-perturbed values, for blocks with the largest edge-length changes, and the corresponding rates of change from the adapted piecewise flat Ricci flow, showing a clear suppression of the exponential instability.
	}
	\label{fig:NumNonExp}
\end{figure}

\section{Conclusion}
\label{sec:Con}

The stability of the piecewise flat Ricci flow has been demonstrated for cubic and skew type triangulations that have been adapted so that the interior of each block is effectively flat. 
In practice, the six tetrahedra in each block are kept, with the length of each interior body-diagonal re-defined to give a zero deficit angle, and therefore a flat interior for each block. This makes the internal angles easier to compute, since the internal geometry of a tetrahdron is entirely determined by the lengths of its edges, and ensures that the edge-lengths alone determine the geometry of each piecewise flat manifold.

A linear stability analysis has verified the exponential instability of the face-diagonal edges seen in numerical simulations for cubic and skew triangulations of flat Euclidean space. 
The linear coefficient matrices have a constant positive value for the sum of the elements in each row, which must therefore be an eigenvalue, and coincides with the largest growth rate seen in the numerical simulations. 
Once the triangulations are adapted, the row-sums of the linear coefficient matrices are all zero, and the numerical simulations are stable. 
As with related types of matrices, it seems reasonable to expect the constant row-sums to give the largest real eigenvalue for each linear coefficient matrix, which would be zero here, in agreement with the smooth Ricci flow \cite{GIKstability}. 
A proof for this has not been found, but direct computation of the eigenvalues for specific triangulations, and numerical simulations for others, support the hypothesis.

While this paper only shows that the adapted triangulations of a flat manifold are stationary for the piecewise flat Ricci flow of \cite{PLCurv}, 
this flow is seen to successfully approximate the smooth Ricci flow in \cite{PLRF} for adapted cubic and skew type triangulations of a variety of different manifolds.
Also, despite re-defining the body-diagonals to have zero deficit angles, computations of the piecewise flat Ricci curvature in \cite{PLRF} are just as accurate at these body-diagonals as they are at the other edges.

\

\noindent
{\bf Data Availability:} \ The \emph{Mathematica} notebooks used for the computations and numerical simulations, and the data generated by these, are available in the Zenodo repository at https://doi.org/10.5281/zenodo.8067524.

\section*{Acknowledgements}

I'd like to thank Robert Sheehan and Chris Mitchell for many helpful discussions which greatly benefited this work.

\bibliography{Ref}
\bibliographystyle{unsrt}

\end{document}